\documentclass[12pt]{article}
\usepackage{amsmath,amsfonts,amssymb,amsthm,latexsym}
\newtheorem{thm}{Theorem}
\newtheorem{defn}[thm]{Definition}
\newtheorem{lem}[thm]{Lemma}
\newtheorem{cor}[thm]{Corollary}
\newtheorem{exam}[thm]{Example}
\newtheorem{prop}[thm]{Proposition}
\newtheorem{rem}[thm]{Remark}
\DeclareMathOperator{\inte}{Int}
\begin{document}

\centerline {\Large\textbf{Remarks on Ultrametrics and Metric-Preserving Functions}}
\vskip.8cm \centerline {\textbf{Prapanpong Pongsriiam}$^a$\footnote{ Corresponding author} and \textbf{Imchit Termwuttipong}$^b$}

\vskip.5cm
\centerline {$^a$Department of Mathematics, Faculty of Science}
\centerline {Silpakorn University, Nakhon Pathom, 73000, Thailand}
\centerline {{\tt prapanpong@gmail.com}}
\centerline {$^b$Department of Mathematics and Computer Science, Faculty of Science}
\centerline {Chulalongkorn University, Bangkok, 10330, Thailand}
\centerline {{\tt  Imchit.T@chula.ac.th}}

\begin{abstract}
Functions whose composition with every metric is a metric are said to be metric-preserving. In this article, we investigate a variation of the concept of metric-preserving functions where metrics are replaced by ultrametrics.
\end{abstract}
\section{Introduction}
\indent Under what conditions on a function $f:[0,\infty)\to [0,\infty)$ is it the case that for every metric space $(X,d)$, $f\circ d$ is still a metric? It is well-known that for any metric $d$, $\frac{d}{1+d}$ and $\min\{1,d\}$ are bounded metrics topologically equivalent to $d$, while $\frac{d}{1+d^2}$ need not be a metric. \\
\indent We call $f:[0,\infty)\to[0,\infty)$ \textit{metric-preserving} if for all metric spaces $(X,d)$, $f\circ d$ is a metric. Therefore the functions $f$ and $g$ given by $f(x) = \frac{x}{1+x}$ and $g(x) = \min\{1,x\}$ are metric-preserving but $h(x) = \frac{x}{1+x^2}$ is not. The concept of metric-preserving functions first appears in Wilson's article \cite{Wi} and is thoroughly investigated by many authors, see for example, \cite{BD, BD1, C, Da, Do2, Do, Do1, DP, DP1, DP2, Ke, PRS, PV, Po, Sr, TO, V} and references therein.\\
\indent However, other important types of distances such as ultrametrics, pseudometrics, pseudodistances \cite{WP1, WP2}, $w$-distances, and $\tau$-distances have not yet been developed in the connection with metric-preserving functions. These distances have many applications in mathematics, see for example, applications of $w$-distances and $\tau$-distances in \cite{KST, KT, MRT, S1, S3, S2, ST}. We will particularly concern with the ultrametrics which arise naturally in the study of $p$-adic numbers and non-archimedean analysis \cite{BGR, DNS}, topology and dynamical system \cite{CRZ, F, MN, Y}, topological algebra \cite{Co}, and theoretical computer science \cite{PR}.\\
\indent In connection with ultrametrics and metric-preserving functions, the problem arises to investigate the properties of the following functions and compare them with those of metric-preserving functions.
\begin{defn}
Let $f:[0,\infty)\to[0,\infty)$. We say that 
\begin{itemize}
\item[(i)] $f$ is ultrametric-preserving if for all ultrametric spaces $(X,d)$, $f\circ d$ is an ultrametric,
\item[(ii)] $f$ is metric-ultrametric-preserving if for all metric spaces $(X,d)$, $f\circ d$ is an ultrametric, and
\item[(iii)] $f$ is ultrametric-metric-preserving if for all ultrametric spaces $(X,d)$, $f\circ d$ is a metric.
\end{itemize}
For convenience, we also let $\mathcal M$ be the set of all metric-preserving functions, $\mathcal U$ the set of all ultrametric-preserving functions, $\mathcal {UM}$ the set of all ultrametric-metric-preserving functions, and $\mathcal {MU}$ the set of all metric-ultrametric-preserving functions.
\end{defn}
\indent We will give some basic definitions and useful results that will be used throughout this article in the next section. We then give properties and characterizations of those functions in Sections \ref{section3}, \ref{section4}, and \ref{section5}. We discuss and give some results on the continuity aspect of those functions in Section \ref{section6}.
\section{Preliminaries and Lemmas}
In this section, we give some basic definitions and results for the convenience of the reader. First, we recall the definition of a metric space and an ultrametric space.\\
\indent A \textit{metric space} is a set $X$ together with a function $d:X\times X\to[0,\infty)$ satisfying the following three conditions:
\begin{itemize}
\item[(M1)] For all $x, y\in X$, $d(x,y)=0$ if and only if $x=y$,
\item[(M2)] for all $x, y\in X$, $d(x,y) = d(y,x)$, and
\item[(M3)] for all $x, y, z\in X$, $d(x,y)\leq d(x,z)+d(z,y)$.
\end{itemize}
An \textit{ultrametric space} is a metric space $(X,d)$ satisfying the stronger inequality (called the ultrametric inequality):
\begin{itemize}
\item[(U3)] For all $x, y, z\in X$, $d(x,y)\leq \max\{d(x,z),d(z,y)\}$.
\end{itemize}
A metric space $(X,d)$ is said to be \textit{topologically discrete} if for every $x\in X$ there is an $\varepsilon > 0$ such that $B_d(x,\varepsilon) = \{x\}$ where $B_d(x,\varepsilon)$ denote the open ball center at $x$ and radius $\varepsilon$. In addition, $(X,d)$ is said to be \textit{uniformly discrete} if there exists an $\varepsilon > 0$ such that $B_d(x,\varepsilon) = \{x\}$ for every $x\in X$.\\ 
\indent Next we recall the definitions concerning certain behaviors of functions. Throughout, we let $f:[0,\infty)\to[0,\infty)$ and let $I\subseteq [0,\infty)$. Then $f$ is said to be \textit{increasing} on $I\subseteq [0,\infty)$ if $f(x)\leq f(y)$ for all $x, y\in I$ satisfying $x < y$, and $f$ is said to be \textit{strictly increasing} on $I\subseteq [0,\infty)$ if $f(x) < f(y)$ for all $x, y\in I$ satisfying $x < y$. The notion of \textit{decreasing} or \textit{strictly decreasing} functions is defined similarly.\\
\indent The function $f$ is said to be \textit{amenable} if $f^{-1}(\{0\}) = \{0\}$, and $f$ is said to be \textit{tightly bounded} on $(0,\infty)$ if there is $v>0$ such that $f(x) \in [v,2v]$ for all $x>0$. We say that $f$ is \textit{subadditive} if $f(a+b)\leq f(a)+f(b)$ for all $a, b\in [0,\infty)$, $f$ is \textit{convex} if $f((1-t)x_1+tx_2) \leq (1-t)f(x_1)+tf(x_2)$ for all $x_1, x_2\in [0,\infty)$ and $t\in [0,1]$, and $f$ is \textit{concave} if $f((1-t)x_1+tx_2) \geq (1-t)f(x_1)+tf(x_2)$ for all $x_1, x_2\in [0,\infty)$ and $t\in [0,1]$. As mentioned earlier, we say that $f$ is \textit{metric-preserving} if for all metric spaces $(X,d)$, $f\circ d$ is a metric. Furthermore, $f$ is \textit{strongly metric-preserving} if $f\circ d$ is a metric equivalent to $d$ for every metric $d$.\\
\indent Now we are ready to state the results which will be applied in the proof of our theorems.
\begin{lem}\label{amensubincreasgotmetric}
Let $f:[0,\infty)\to[0,\infty)$. If $f$ is amenable, subadditive and increasing on $[0,\infty)$, then $f$ is metric-preserving.
\end{lem}   
\begin{proof}
The proof can be found, for example, in \cite{C, Do2}.
\end{proof}
\begin{lem}\label{amentightgotmetric}
If $f:[0,\infty)\to[0,\infty)$ is amenable and tightly bounded, then $f$ is metric-preserving.
\end{lem}
\begin{proof}
The proof can be found, for example, in \cite{BD1, C}.
\end{proof}
The next lemma might be less well-known, so we give a proof here for completeness.
\begin{lem}\label{amencavegotdecreas}
If $f:[0,\infty)\to[0,\infty)$ is amenable and concave, then the function $x\mapsto \frac{f(x)}{x}$ is decreasing on $(0,\infty)$
\end{lem}
\begin{proof}
Let $a, b\in (0,\infty)$ and $a<b$. Since $f$ is concave, we obtain
$$
f(a) = f\left(\left(1-\frac{a}{b}\right)(0)+\left(\frac{a}{b}\right)(b)\right)\geq \left(1-\frac{a}{b}\right)f(0)+\frac{a}{b}f(b) = \frac{a}{b}f(b). 
$$
Therefore $\frac{f(a)}{a} \geq \frac{f(b)}{b}$, as desired.
\end{proof}
\begin{lem}\label{newtypelemma1}
Let $(X,d)$ be an ultrametric space. Then for every $x_1, x_2, \ldots, x_n\in X$,
$$
d(x_1,x_n) \leq \max\left\{d(x_1,x_2), d(x_2,x_3),\ldots, d(x_{n-1},x_n)\right\}.
$$
\end{lem}
\begin{proof}
We have
\begin{align*}
d(x_1,x_n) &\leq \max\left\{d(x_1,x_2), d(x_2,x_n)\right\}\\
&\leq \max\left\{d(x_1,x_2), \max\left\{d(x_2,x_3), d(x_3,x_n)\right\}\right\}\\
&= \max\left\{d(x_1,x_2), d(x_2,x_3), d(x_3,x_n)\right\}.
\end{align*}
A repeated application of the ultrametric inequality as above gives the desired result.
\end{proof}
Next we give basic relations and properties of the functions in $\mathcal M$, $\mathcal U$, $\mathcal{MU}$, and $\mathcal{UM}$.
\begin{prop}\label{newtypeprop1}
$\mathcal {MU}\stackrel{\text{(S1)}}{\subseteq} \mathcal{U}\cap \mathcal{M}\stackrel{\text{(S2)}}{\subseteq} \mathcal{U}, \mathcal{M}\stackrel{\text{(S3)}}{\subseteq}\mathcal{U}\cup \mathcal{M}\stackrel{\text{(S4)}}{\subseteq} \mathcal{UM}$.
\end{prop}
\begin{proof}
Since an ultrametric is a metric, $\mathcal{MU}\subseteq \mathcal {U}$ and $\mathcal{MU}\subseteq \mathcal {M}$. So (S1) follows. Similarly, $\mathcal{U}\subseteq \mathcal {UM}$ and $\mathcal{M}\subseteq \mathcal {UM}$, so (S4) holds. (S2) and (S3) are true in general. 
\end{proof} 
We will obtain characterization of the functions in $\mathcal{U}$, $\mathcal{MU}$ and $\mathcal{UM}$ in later section. Then we will show that the relation $\subseteq$ in Proposition \ref{newtypeprop1} is in fact a proper subset. It is easy to see that if $f\in \mathcal M$, then $f$ is amenable. We extend this to the case of any function $f\in \mathcal M\cup \mathcal{MU}\cup\mathcal {UM}\cup \mathcal U$.
\begin{prop}\label{amenablelemma}
If $f\in \mathcal {UM}$, then $f$ is amenable.
\end{prop}
\begin{proof}
Assume that $f\in \mathcal {UM}$. To show that $f$ is amenable, we let $x\in [0,\infty)$ be such that $f(x)=0$. Let $X = \{A, B, C\}\subseteq \mathbb R^2$ where $A = \left(-\frac{x}{2},0\right)$, $B = \left(\frac{x}{2},0\right)$, and $C = \left(0,\frac{\sqrt 3x}{2}\right)$. Let $d_2$ be the Euclidean metric on $\mathbb R^2$ and let $d = d_2\mid_X$ be the restriction of $d_2$ on $X$. Then $d(A,B) = d(A,C) = d(B,C) = x$. Therefore $(X,d)$ is an ultrametric space. So $f\circ d$ is a metric on $X$. Now $f(0) = f(d(A,A)) = (f\circ d)(A,A) = 0$, and $(f\circ d)(A,B) = f(d(A,B)) = f(x) = 0$, which implies $A = B$. That is $\left(-\frac{x}{2},0\right) = \left(\frac{x}{2},0\right)$. Hence $x=0$. This shows that $f$ is amenable as desired.
\end{proof}
\begin{cor}\label{corafteramenablelemma}
If $f:[0,\infty)\to [0,\infty)$ is in $\mathcal M$, $\mathcal{MU}$, $\mathcal U$, or $\mathcal{UM}$, then $f$ is amenable.
\end{cor}
\begin{proof}
By Proposition \ref{newtypeprop1}, $\mathcal M \cup\mathcal{MU}\cup \mathcal U \cup\mathcal{UM} = \mathcal{UM}$. So the result follows from Proposition \ref{amenablelemma}.
\end{proof}
\section{Ultrametric-Preserving Functions}\label{section3}
In this section, we obtain characterizations of ultrametric-preserving functions. Then we compare their properties with those of metric-preserving functions.
\begin{thm}\label{ultraiffamenableincreas}
Let $f:[0,\infty)\to[0,\infty)$. Then $f$ is ultrametric-preserving if and only if $f$ is amenable and increasing.
\end{thm}
\begin{proof}
Assume that $f$ is ultrametric-preserving. By Corollary \ref{corafteramenablelemma}, it suffices to show that $f$ is increasing. Let $a, b\in [0,\infty)$ and $a<b$. Let $d_2$ be the Euclidean metric on $\mathbb R^2$ and let $X = \{A, B, C\}\subseteq \mathbb R^2$ where $A = \left(-\frac{a}{2},0\right)$, $B = \left(\frac{a}{2},0\right)$, and $C = \left(0,\sqrt{\frac{4b^2-a^2}{4}}\right)$. Let $d = d_2\mid_X$ be the restriction of $d_2$ on $X$. Then $d(A,B) = a$, $d(A,C) = d(B,C) = b$. Therefore $(X,d)$ is an ultrametric space. Since $f$ is ultrametric-preserving, $f\circ d$ is an ultrametric. Therefore 
$$
f(a) = f\circ d(A,B) \leq \max\{f\circ d(A,C), f\circ d(B,C)\} = f(b),
$$
as required. Next assume that $f$ is increasing and amenable. Let $(X,d)$ be an ultrametric space, and let $x, y, z\in X$. Since $f$ is amenable, $f\circ d(x,y) = 0$ if and only if $x=y$. Since $d$ is an ultrametric, $d(x,z)\leq \max\{d(x,y), d(y,z)\}$. So $d(x,z)\leq d(x,y)$ or $d(x,z)\leq d(y,z)$. If $d(x,z)\leq d(x,y)$, then $f(d(x,z))\leq f(d(x,y)) \leq \max\{f\circ d(x,y), f\circ d(y,z)\}$. If $d(x,z)\leq d(y,z)$, then $f(d(x,z))\leq f(d(y,z)) \leq \max\{f\circ d(x,y), f\circ d(y,z)\}$. In any case $f\circ d(x,z) \leq \max\{f\circ d(x,y), f\circ d(y,z)\}$. Therefore $f\circ d$ is an ultrametric. This completes the proof.
\end{proof}
\begin{cor}
Let $f:[0,\infty)\to[0,\infty)$. Then the following statements hold:
\begin{itemize}
\item[(i)] If $f$ is ultrametric-preserving and subadditive, then $f$ is metric-preserving.
\item[(ii)] If $f$ is metric-preserving and increasing on $[0,\infty)$, then $f$ is ultrametric-preserving.
\end{itemize}
\end{cor}
\begin{proof}
We obtain that (i) follows from Theorem \ref{ultraiffamenableincreas} and Lemma \ref{amensubincreasgotmetric}, and (ii) follows from Corollary \ref{corafteramenablelemma} and Theorem \ref{ultraiffamenableincreas}.
\end{proof}
The next example shows that $\mathcal M\nsubseteq \mathcal U$ and $\mathcal U\nsubseteq \mathcal M$.
\begin{exam}\label{exam11sectionultrametricpreserving}
Let $f, g:[0,\infty)\to[0,\infty)$ be given by 
$$
f(x) = x^2\quad\text{and}\quad g(x) = \begin{cases}
0, &\text{if $x=0$};\\
1, &\text{if $x\in \mathbb Q-\{0\}$};\\
2, &\text{if $x\in \mathbb Q^c$}.
\end{cases} 
$$
\end{exam}
By Theorem \ref{ultraiffamenableincreas}, $f$ is ultrametric-preserving and $g$ is not ultrametric-preserving. If $d$ is the usual metric on $\mathbb R$, we see that 
$$
f\circ d(1,3) = f(2) = 4 > 2 = f\circ d(1,2) + f\circ d(2,3).
$$
So $f\circ d$ is not a metric and therefore $f$ is not metric-preserving. Since $g(x)\in [1,2]$ for all $x>0$, $g$ is tightly bounded and therefore, by Lemma \ref{amentightgotmetric}, $g$ is metric-preserving. In conclusion, $f\in \mathcal U$, $f\notin \mathcal M$, $g\in \mathcal M$, and $g\notin \mathcal U$. This shows that $\mathcal U\nsubseteq \mathcal M$ and $\mathcal M\nsubseteq \mathcal U$. This example also shows that the relations (S2) and (S3) in Proposition \ref{newtypeprop1} are proper subsets.\\
Next we give some results concerning concavity of the functions in $\mathcal U\cup \mathcal M$.
\begin{thm}\label{amenconcavegotultra}
Let $f:[0,\infty)\to[0,\infty)$. If $f$ is amenable and concave, then $f$ is ultrametric-preserving.
\end{thm}
\begin{proof}
Assume that $f$ is amenable and concave. We will show that $f$ is increasing. First observe that if $y>0$, then $f(y) > f(0)$ because $f$ is amenable. Next let $y>x>0$ and suppose for a contradiction that $f(y) < f(x)$. Let $t = \frac{f(y)}{f(x)}$, $x_1 = \frac{yf(x)-xf(y)}{f(x)-f(y)}$, and $x_2 = x$. Then $t\in (0,1)$, and $x_1, x_2\in (0,\infty)$. Since $f$ is concave, we obtain
$$
f(y) = f((1-t)x_1+tx_2) \geq (1-t)f(x_1)+tf(x_2) = (1-t)f(x_1)+f(y).
$$
This implies that $f(x_1) = 0$ which contradicts the fact that $x_1>0$ and $f$ is amenable. Hence $f$ is increasing on $[0,\infty)$. By Theorem \ref{ultraiffamenableincreas}, $f$ is ultrametric-preserving.
\end{proof}
\begin{cor}\label{cor12metricpreserveultra}
If $f:[0,\infty)\to[0,\infty)$ is amenable and concave, then $f$ is both ultrametric-preserving and metric-preserving.
\end{cor}
\begin{proof}
The first part comes from Theorem \ref{amenconcavegotultra}. The other part have appeared in the literature but we will give an alternative proof here. We know that $f$ is increasing by Theorem \ref{amenconcavegotultra} and Theorem \ref{ultraiffamenableincreas}. So by Lemma \ref{amensubincreasgotmetric}, it suffices to show that $f$ is subadditive. Let $a, b\in (0,\infty)$. By Lemma \ref{amencavegotdecreas}, we have $\frac{f(a+b)}{a+b}\leq \min \left\{\frac{f(a)}{a}, \frac{f(b)}{b}\right\}$. Therefore
$$
f(a+b) = a\left(\frac{f(a+b)}{a+b}\right)+b\left(\frac{f(a+b)}{a+b}\right) \leq a\frac{f(a)}{a}+b\frac{f(b)}{b} = f(a)+f(b),
$$
as required. This completes the proof.
\end{proof}
The next example shows that there exists a function which is both metric-preserving and ultrametric-preserving but not concave.
\begin{exam}\label{exam14sectionultrametricpreserving}
Let $f:[0,\infty)\to[0,\infty)$ be defined by 
$$
f(x) = \begin{cases}
x, &x\in [0,1];\\
1, &x\in [1,10];\\
x-9, &x\in (10,11);\\
2, &x\geq 11.
\end{cases}
$$
\end{exam}
It is easy to see that $f$ is amenable and increasing. So, by Theorem \ref{ultraiffamenableincreas}, $f$ is ultrametric-preserving. Next we will show that $f$ is metric-preserving. By Lemma \ref{amensubincreasgotmetric}, it suffices to show that $f$ is subadditive. Observe that $f(x)\leq x$ and $f(x)\leq 2$ for every $x\in [0,\infty)$. We consider $a, b \in [0,\infty)$ in several cases.
\begin{itemize}
\item[] If $a, b\in [0,1]$, then $f(a)+f(b) = a+b\geq f(a+b)$.
\item[] If $a, b\in [1,10]$, then $f(a)+f(b)  =2 \geq f(a+b)$. \\
Similarly, if $a, b > 10$, then $f(a)+f(b) > 2 \geq f(a+b)$.
\item[] If $a\in [0,1]$, $b\in [1,10]$, then 
$$
f(a)+f(b)=a+1\geq\max\{1, a+b-9\} \geq f(a+b).
$$
\item[] If $a\in [0,1]$, $b\in[10,11]$, then $f(a)+f(b) = a+b-9 \geq f(a+b)$.
\item[] If $a\in [0,1]$, $b\in [11,\infty)$, then $f(a)+f(b)=a+2\geq 2 = f(a+b)$.
\item[] If $a\in [1,10]$, $b\in [10,\infty)$, then $f(a)+f(b)=b-8\geq 2 = f(a+b)$.
\end{itemize}
The other cases can be obtained similarly. Therefore $f$ is subadditive. Hence $f$ is metric-preserving. But $f\left(\frac{9+11}{2}\right) < \frac{f(9)+f(11)}{2}$, so $f$ is not concave. That is, $f\in \mathcal{U}\cap \mathcal M$ but $f$ is not concave. In addition, $f$ is not a constant on $(0,\infty)$. So this example also shows that $\mathcal{U\cap M}\nsubseteq \mathcal{MU}$ and the relation (S1) in Proposition \ref{newtypeprop1} is a proper subset.
\section{Metric-Ultrametric-Preserving Functions}\label{section4}
In this section, we characterize the functions in $\mathcal{MU}$. We will see that this notion is so strong that it forces the functions to be a constant on $(0,\infty)$. More precisely, we obtain the following theorem.
\begin{thm}\label{thm14metricultraiffamen}
Let $f:[0,\infty)\to[0,\infty)$. Then $f$ is metric-ultrametric-preserving if and only if $f$ is amenable and $f$ is a constant on $(0,\infty)$.
\end{thm}
\begin{proof}
First assume that $f$ is amenable and is a  constant on $(0,\infty)$. That is there exists a constant $c>0$ such that $f(x) = \begin{cases}
0, &\text{if $x=0$};\\
c, &\text{if $x>0$}.
\end{cases}$
To show that $f$ is metric-ultrametric-preserving, let $(X,d)$ be a metric space and let $x, y, z\in X$. If $x=y$ or $x=z$ or $y=z$, then it is easy to see that $f\circ d(x,y)\leq \max\{f\circ d(x,z), f\circ d(z,y)\}$. If $x, y, z$ are all distinct, then $f\circ d(x,y) = c = f\circ d(x,z) = f\circ d(y,z)$ and therefore 
$$
f\circ d(x,y) \leq \max\{f\circ d(x,z), f\circ d(z,y)\}.
$$
This shows that $f\circ d$ is an ultrametric. In the other direction, we assume that $f\in \mathcal{MU}$. By Corollary \ref{corafteramenablelemma}, it is enough to show that $f$ is a constant on $(0,\infty)$. Throughout the proof, we let $d$ be the usual metric on $\mathbb R$ and $d_2$ the Euclidean metric on $\mathbb R^2$. We will apply Lemma \ref{newtypelemma1} repeatedly. First we will show that 
$$
f(1) = f\left(\frac{1}{n}\right) = f\left(\frac{m}{n}\right)\quad\text{for every $m, n\in\mathbb N$}.
$$
So we let $m, n\in\mathbb N$ be arbitrary. Since $f\in \mathcal {MU}$, $f\circ d$ is an ultrametric on $\mathbb R$. By Lemma \ref{newtypelemma1}, we have
\begin{align*}
f(1) &= f\circ d(0,1)\\
&\leq \max\left\{f\circ d\left(0,\frac{1}{n}\right), f\circ d\left(\frac{1}{n},\frac{2}{n}\right),\ldots, f\circ d\left(\frac{n-1}{n},1\right)\right\}\\
&= \max\left\{f\left(\frac{1}{n}\right),f\left(\frac{1}{n}\right),\ldots, f\left(\frac{1}{n}\right)\right\} = f\left(\frac{1}{n}\right).
\end{align*} 
Next let $A = \left(-\frac{1}{2n},0\right)$, $B = \left(\frac{1}{2n},0\right)$, $C = \left(0,\sqrt{\frac{4-\left(\frac{1}{n}\right)^2}{4}}\right)$ be points in $\mathbb R^2$. Since $f\in \mathcal{MU}$, $f\circ d_2$ is an ultrametric on $\mathbb R^2$. Therefore 
\begin{align*}
f\left(\frac{1}{n}\right) &= f\circ d_2 (A,B)\leq\max\left\{f\circ d_2(A,C), f\circ d_2(C,B)\right\}\\
&= \max\left\{f(1), f(1)\right\} = f(1).
\end{align*} 
Therefore $f(1) = f\left(\frac{1}{n}\right)$. By a similar method, we obtain
\begin{align*}
f\left(\frac{m}{n}\right) &= f\circ d\left(0,\frac{m}{n}\right) \\
&\leq \max\left\{\left.f\circ d\left(\frac{k-1}{n},\frac{k}{n}\right)\right| k\in \{1, 2, \ldots, m\}\right\} = f\left(\frac{1}{n}\right).
\end{align*}
In addition, we let $A = \left(-\frac{1}{2n},0\right)$, $B = \left(\frac{1}{2n},0\right)$, $C = \left(0,\sqrt{\frac{4\left(\frac{m}{n}\right)^2-\left(\frac{1}{n}\right)^2}{4}}\right)$ be points in $\mathbb R^2$ so that
\begin{align*}
f\left(\frac{1}{n}\right) = f\circ d_2 (A,B) \leq \max\left\{f\circ d_2(A,C), f\circ d_2(C,B)\right\} = f\left(\frac{m}{n}\right).
\end{align*}
Therefore $f\left(\frac{m}{n}\right) = f\left(\frac{1}{n}\right)$. Hence $f\left(\frac{m}{n}\right) = f\left(\frac{1}{n}\right) = f(1)$ for every $m, n\in \mathbb N$, as asserted. We conclude that 
\begin{equation}\label{metricultrastep2eq1}
f(q) = f(1)\quad\text{for every $x\in \mathbb Q\cap (0,\infty)$}.
\end{equation}
Next let $a \in \mathbb Q^c\cap (0,\infty)$. We will show that $f(a) = f(1)$. Let $q_1, q_2\in \mathbb Q\cap (0,\infty)$ be such that $q_1 < a < q_2$. Let $A_1 = \left(-\frac{q_1}{2},0\right)$, $B_1 = \left(\frac{q_1}{2},0\right)$, $C_1 = \left(0,\sqrt{\frac{4a^2-q_1^2}{4}}\right)$, $A_2 = \left(-\frac{a}{2},0\right)$, $B_2 = \left(\frac{a}{2},0\right)$, $C_2 = \left(0,\sqrt{\frac{4q_2^2-a^2}{4}}\right)$ be points in $\mathbb R^2$. By (\ref{metricultrastep2eq1}) and the fact that $f\circ d_2$ is an ultrametric on $\mathbb R^2$, we obtain 
\begin{align*}
f(1) &= f(q_1) = f\circ d_2(A_1, B_1)\\
&\leq \max\left\{f\circ d_2(A_1,C_1), f\circ d_2(C_1,B_1)\right\}\\
&= f(a) = f\circ d_2(A_2,B_2)\\
&\leq \max\left\{f\circ d_2(A_2,C_2), f\circ d_2(C_2,B_2)\right\}\\
&= f(q_2) = f(1).
\end{align*}
This shows that 
\begin{equation}\label{metricultrastep2eq2}
f(a) = f(1)\quad\text{for all $a\in \mathbb Q^c\cap (0,\infty)$}.
\end{equation}
From (\ref{metricultrastep2eq1}) and (\ref{metricultrastep2eq2}), we see that $f(x) = f(1)$ for all $x\in (0,\infty)$. This completes the proof. 
\end{proof}
\indent Let $f$ be a metric-preserving function and let $d$ be a metric. Then either $f\circ d$ is a metric equivalent to $d$ or $f\circ d$ is a uniformly discrete metric \cite{BD1, Do2}. In addition, $f$ is continuous on $[0,\infty)$ if and only if it is continuous at 0 \cite{BD1, C, Do2}. But by Theorem \ref{thm14metricultraiffamen}, every metric-ultrametric-preserving function $f$ is always discontinuous at 0 and $f\circ d$ is always a uniformly discrete metric for all metric $d$. We record this in the next corollary.
\begin{cor}\label{cor16sectionmetricultrapreserving}
Let $f:[0,\infty)\to[0,\infty)$ be metric-ultrametric-preserving. Then
\begin{itemize}
\item[(i)] $f\circ d$ is a uniformly discrete metric for every metric $d$,
\item[(ii)] $f$ is discontinuous at $0$ and is continuous on $(0,\infty)$.
\end{itemize}
\end{cor}
\begin{proof}
By Theorem \ref{thm14metricultraiffamen}, there exists $c>0$ such that $f(x) = \begin{cases}
0,\quad\text{if $x=0$};\\
c,\quad\text{if $x>0$}.
\end{cases}$
So (ii) follows immediately. If $(X,d)$ is a metric space, then 
$$
f\circ d(x,y) = 
\begin{cases}
0,\quad\text{if $x=y$};\\
c,\quad\text{if $x\neq y$}.
\end{cases}
$$
So if we let $\varepsilon = \frac{c}{2}$, then $B_{f\circ d}(x,\varepsilon) = \{x\}$ for every $x\in X$. This proves (i).
\end{proof}
\section{Ultrametric-Metric-Preserving Functions}\label{section5}
In this section, we give a characterization of the functions in $\mathcal{UM}$ in terms of special type of triangle triplets. Recall that a triple $(a,b,c)$ of nonnegative real numbers is called \textit{triangle triplet} if $a\leq b+c$, $b\leq c+a$, and $c\leq a+b$. We denote by $\Delta$ the set of all triangle triplets. We introduce a special type of triangle triplets that will be used to characterize ultrametric-metric-preserving functions in the next definition.
\begin{defn}
A triple $(a,b,c)$ of nonnegative real numbers will be called \textit{ultra triangle triplet} if $a\leq \max\{b,c\}$, $b\leq \max\{c,a\}$, and $c\leq \max\{a,b\}$. We denote by $\Delta_\infty$ the set of all ultra triangle triplets.
\end{defn}
Since we will compare the functions $f$ in $\mathcal {UM}$ with those in $\mathcal M$, we first state a characterization of metric-preserving functions in terms of triangle triplets.
\begin{thm}\label{thm15ultrametricpreserv}
Let $f:[0,\infty)\to[0,\infty)$ be amenable. Then the following statements are equivalent:
\begin{itemize}
\item[(i)] $f$ is metric-preserving,
\item[(ii)] for each $(a,b,c)\in \Delta$, $(f(a),f(b),f(c))\in \Delta$,
\item[(iii)] for each $(a,b,c)\in \Delta$, $f(a) \leq f(b)+f(c)$.
\end{itemize}
\end{thm}
\begin{proof}
The proof can be found, for example, in \cite{BD1, C, Do2}.
\end{proof}
Similar to Theorem \ref{thm15ultrametricpreserv}, we obtain a characterization of the functions in $\mathcal{UM}$ in terms of ultra triangle triplets as follows.
\begin{thm}\label{ultrametricthm}
Let $f:[0,\infty)\to[0,\infty)$ be amenable. Then the following statements are equivalent:
\begin{itemize}
\item[(i)] $f$ is ultrametric-metric-preserving,
\item[(ii)] for each $(a,b,c)\in \Delta_\infty$, $(f(a),f(b),f(c))\in \Delta$,
\item[(iii)] for each $0\leq a\leq b$, $f(a) \leq 2f(b)$.
\end{itemize}
\end{thm}
To prove Theorem \ref{ultrametricthm}, the following lemmas are useful.
\begin{lem}\label{ultrametricthmlemma1}
If $(X,d)$ is an ultrametric space and $x, y, z\in X$, then the triple $(d(x,y),d(x,z),d(z,y))$ is an ultra triangle triplet. Conversely, if $(a,b,c)$ is an ultra triangle triplet, then there exist an ultrametric space $(X,d)$ and $x, y, z\in X$ such that $(a,b,c) = (d(x,y),d(x,z),d(z,y))$.
\end{lem}
\begin{lem}\label{ultrametricthmlemma2}
If $(a,b,c)\in\Delta_\infty$, then 
$$
\text{(i) $a\leq b = c$ or (ii) $b\leq c = a$ or (iii) $c\leq a = b$}.
$$
\end{lem}
We will prove Lemma \ref{ultrametricthmlemma2}, Lemma \ref{ultrametricthmlemma1}, and then Theorem \ref{ultrametricthm}, respectively.\vspace{0.3cm}\\
\noindent\textbf{Proof of Lemma \ref{ultrametricthmlemma2}.}\vspace{0.25cm}\\
Let $(a,b,c)\in \Delta_\infty$. Suppose that $a, b, c$ are all distinct. Without loss of generality, we can assume that $a<b<c$. Then $c > \max\{a, b\}$ which contradicts the fact that $(a,b,c)\in \Delta_\infty$. So $a, b, c$ are not all distinct. If $a = b$, then $c\leq \max\{a, b\} = a$ and (iii) holds. Similarly, if $a=c$, then (ii) holds and if $b = c$, then (i) holds.\qed\vspace{0.3cm}\\
\noindent\textbf{Proof of Lemma \ref{ultrametricthmlemma1}.} \vspace{0.25cm}\\
The first part follows immediately from the ultrametric inequality of $d$. For the converse, we let $(a,b,c)\in \Delta_\infty$. By Lemma \ref{ultrametricthmlemma2}, we can assume that $a\leq b = c$ (the other cases can be proved similarly). Let $X = \{A, B, C\}\subseteq \mathbb R^2$ where $A = \left(-\frac{a}{2},0\right)$, $B = \left(\frac{a}{2},0\right)$, and $C = \left(0,\sqrt{\frac{4b^2-a^2}{4}}\right)$. Let $d_2$ be the Euclidean metric on $\mathbb R^2$ and $d = d_2\mid_X$. Then $(X,d)$ is an ultrametric space and $(a,b,c) = (d_2(A,B), d_2(A,C), d_2(C,B))$.\qed\vspace{0.3cm}\\
\noindent\textbf{Proof of Theorem \ref{ultrametricthm}.}  \vspace{0.25cm}\\
\indent (i)$\to$(ii) Let $f\in \mathcal{UM}$ and let $(a,b,c)\in \Delta_\infty$. Then by Lemma \ref{ultrametricthmlemma1}, there exist an ultrametric space $(X,d)$ and $x, y, z\in X$ such that 
$$
(a,b,c) = (d(x,y),d(x,z),d(z,y)).
$$
 Since $f\in \mathcal{UM}$, $(X,f\circ d)$ is a metric space. It follows from the triangle inequality of $f\circ d$ that $(f\circ d(x,y),f\circ d(x,z),f\circ d(z,y))$ is a triangle triplet. That is $(f(a),f(b),f(c))\in \Delta$.\vspace{0.3cm}\\
\indent (ii)$\to$(iii) Assume that (ii) holds. Let $0\leq a\leq b$. Then $(a,b,b)\in \Delta_\infty$. So $(f(a),f(b),f(b))\in \Delta$ by (ii). Therefore $f(a)\leq f(b)+f(b) = 2f(b)$, as required.\vspace{0.3cm}\\
\indent (iii)$\to$(i) Assume that (iii) holds. Let $(X,d)$ be an ultrametric space. Since $f$ is amenable, $f\circ d(x,y) = 0$ if and only if $x=y$. So it remains to show that the triangle inequality holds for $f\circ d$. Let $x, y, z\in X$. Then by Lemma \ref{ultrametricthmlemma1}, $(d(x,y),d(x,z),d(z,y))\in \Delta_\infty$. Then by Lemma \ref{ultrametricthmlemma2}, we can assume that $d(x,y)\leq d(x,z) = d(z,y)$ (the other cases can be proved similarly). Then by (iii), we obtain 
\begin{align*}
f\circ d(x,y) = f(d(x,y)) &\leq 2f(d(x,z))\\
&= f(d(x,z))+f(d(z,y))\\
&= f\circ d(x,z)+f\circ d(z,y),\quad\text{as required}.
\end{align*}
Hence the proof is complete. \qed\vspace{0.3cm}\\
\indent Next we give an example to show that the relation (S4) in Proposition \ref{newtypeprop1} is a proper subset.
\begin{exam}\label{example24fdisnotmetric}
Let $f:[0,\infty)\to[0,\infty)$ be given by 
$$
f(x) = \begin{cases}
x, &\text{if $x\leq 1$};\\
\frac{1}{2}, &\text{if $x>1$}.
\end{cases}
$$
\end{exam}
Let $d$ be the usual metric on $\mathbb R$. Then 
$$
f\circ d(1,2) = f(1) = 1 > \frac{1}{3}+\frac{1}{2} = f\circ d\left(1,\frac{2}{3}\right)+f\circ d\left(\frac{2}{3},2\right). 
$$
So $f\circ d$ is not a metric and therefore $f\notin \mathcal M$. Since $f$ is not increasing, $f\notin \mathcal U$. Next we will show that $f\in \mathcal{UM}$, by applying Theorem \ref{ultrametricthm}. Let $0\leq a\leq b$. If $b\geq \frac{1}{2}$, then $f(b) \geq \frac{1}{2}$ and therefore $2f(b) \geq 1\geq f(x)$ for all $x\in [0,\infty)$. In particular, $2f(b) \geq f(a)$. If $b < \frac{1}{2}$, then $a < \frac{1}{2}$ and thus $f(a) = a \leq b = f(b) \leq 2f(b)$. In any case, we have $f(a)\leq 2f(b)$. Hence $f\in \mathcal {UM}$ but $f\notin \mathcal M$ and $f\notin \mathcal U$. This example shows that $\mathcal{UM}\nsubseteq \mathcal{U}\cup\mathcal{M}$ and the relation (S4) in Proposition \ref{newtypeprop1} is in fact a proper subset.
\begin{rem}
\begin{itemize}
\item[]
\item[1)] From Example \ref{exam11sectionultrametricpreserving}, Example \ref{exam14sectionultrametricpreserving}, and Example \ref{example24fdisnotmetric}, we now see that the relations (S1), (S2), (S3), and (S4) in Proposition \ref{newtypeprop1} are in fact proper subsets.
\item[2)] If we replace $\frac{1}{2}$ in the definition of $f$ in Example \ref{example24fdisnotmetric} by a constant $c$ (that is $f(x) = x$ if $x\leq 1$ and $f(x) = c$ if $x>1$), then $f\in \mathcal {UM}$ if and only if $c \geq \frac{1}{2}$.
\end{itemize}
\end{rem}
\section{Continuity}\label{section6}
In this section, we investigate the continuity aspect of the functions in $\mathcal M$, $\mathcal U$, $\mathcal{UM}$, and $\mathcal{MU}$. By Corollary \ref{cor16sectionmetricultrapreserving}, the continuity of metric-ultrametric-preserving functions is trivial: they are always discontinuous at $0$ and continuous else where. The continuity of metric-preserving functions has also been investigated by many authors \cite{BD, BD1, C, Do2, Do1, V, Wi} but we can still extend it further in the next theorem. \\
\indent Before we state the theorem, let us recall some definitions concerning generalized continuities. Let $f:[0,\infty)\to[0,\infty)$. Then $f$ is said to be \textit{weakly continuous} at $a\neq 0$ if and only if there are sequences $(x_n)$ and $(y_n)$ such that $(x_n)$ is strictly increasing and converges to $a$, $(y_n)$ is strictly decreasing and converges to $a$, and $f(x_n)$ and $f(y_n)$ converge to $f(a)$. If $a=0$, then $f$ is said to be \textit{weakly continuous} at $a$ if and only if there exists a strictly decreasing sequence $(y_n)$ converging to $a$ such that $f(y_n)$ converges to $f(a)$. We refer the reader to \cite{P} for weak continuity of functions defined on a more general domain. \\
\indent  Unlike weak continuity, quasi continuity and almost continuity seem to be first given in a more general domain than a subset of $\mathbb R$. So we let $X$ and $Y$ be topological spaces and let $g:X\to Y$. Then $g$ is said to be \textit{quasi continuous} at $a\in X$ if for all open sets $U$ of $X$ and $V$ of $Y$ such that $a\in U$ and $f(a)\in V$, there is a nonempty open sets $G$ of $X$ such that $G\subseteq U$ and $f(G)\subseteq V$. The function $g$ is said to be \textit{almost continuous at $x$ in the sense of Singal} (briefly a.c.S. at $x$) if for each open set $V$ of $Y$ containing $f(x)$, there exists an open set $U$ containing $x$ such that $f(U)\subseteq \inte(\overline V)$ and $g$ is said to be \textit{almost continuous at $x$ in the sense of Husain} (briefly a.c.H. at $x$) if for each open set $V$ of $Y$ containing $f(x)$, $\overline{f^{-1}(V)}$ is a neighborhood of $x$. The function $g$ is said to be \textit{quasi continuous} on $A\subseteq X$ (or a.c.S. on $A$, or a.c.H. on $A$) if it is quasi continuous at every $a\in A$ (a.c.S. at $a$ for every $a\in A$, a.c.H. at $a$ for every $a\in A$). 
\begin{rem}
\item[1)] The concept of a.c.S. functions and a.c.H. functions are not equivalent as shown by Long and Carnahan \cite{LC}.
\item[2)] There are several other types of continuities in the literature. Some of them have the same name but different definition, see \cite{L} for instance, a different definition of weak continuity. We refer the reader to \cite{CR, E, M, NP, PLB} and the other references for additional details and information. 
\end{rem}
Now we are ready to state our theorem. We will see that there is a similarity and dissimilarity between continuity of the functions in $\mathcal M$ and $\mathcal{UM}$.
\begin{thm}\label{newthmcontinuitythm}
Let $f:[0,\infty)\to[0,\infty)$ be metric-preserving. The the following statements are equivalent:
\begin{itemize}
\item[1)] $f$ is continuous at $[0,\infty)$,
\item[2)] $f$ is continuous at $0$,
\item[3)] For every $\varepsilon > 0$, there exists and $x>0$ such that $f(x) < \varepsilon$,
\item[4)] $f$ is strongly metric-preserving,
\item[5)] $f$ is uniformly continuous on $[0,\infty)$,
\item[6)] $f$ is weakly continuous on $[0,\infty)$,
\item[7)] $f$ is weakly continuous at $0$,
\item[8)] $f$ is quasi continuous on $[0,\infty)$,
\item[9)] $f$ is quasi continuous at $0$,
\item[10)] $f$ is a.c.S on $[0,\infty)$,
\item[11)] $f$ is a.c.S at $0$,
\item[12)] $f$ is a.c.H on $[0,\infty)$,
\item[13)] $f$ is a.c.H at $0$.
\end{itemize}
\end{thm} 
\begin{proof}
The equivalence of 1), 2), 3), and 4) is proved in \cite{C, Do2}. With a bit more observation, we can prove that 1) to 11) are all equivalent. First we notice that
\begin{equation}\label{newthmcontinuityeq1}
\left|f(a)-f(b)\right| \leq f\left(\left|a-b\right|\right)\quad\text{for all $a, b\in [0,\infty)$}.
\end{equation} 
To prove (\ref{newthmcontinuityeq1}), we let $a, b\in [0,\infty)$. Then $\left(a, b, \left|a-b\right|\right)$ is a triangle triplet. So by Theorem \ref{thm15ultrametricpreserv}, $\left(f(a), f(b), f\left(\left|a-b\right|\right)\right)$ is a triangle triplet. Therefore 
$$
f(a)\leq f(b)+f\left(\left|a-b\right|\right)\quad\text{and}\quad f(b)\leq f(a)+f\left(\left|a-b\right|\right).
$$
Thus $\left|f(a)-f(b)\right|\leq f\left(\left|a-b\right|\right)$, as asserted. Now we will prove that 2), 5), 6), 7), and 3) are equivalent.\vspace{0.3cm}\\
\indent 2)$\to$ 5) Assume that $f$ is continuous at $0$. Let $\varepsilon > 0$. Then there exists a $\delta>0$ such that 
\begin{equation}\label{newthmcontinuityeq2}
\text{if $a\in [0,\delta)$, then $f(a) < \varepsilon$}.
\end{equation}
Now if $x, y\in [0,\infty)$ and $|x-y| < \delta$, then by (\ref{newthmcontinuityeq1}) and (\ref{newthmcontinuityeq2}), we obtain
$$
\left|f(x)-f(y)\right| \leq f\left(\left|x-y\right|\right) < \varepsilon. 
$$
This shows that $f$ is uniformly continuous on $[0,\infty)$. \vspace{0.3cm}\\
\indent It is easy to see that 5) implies 6) and 6) implies 7). \vspace{0.3cm}\\
\indent 7)$\to$3) We assume that 7) holds. Let $(x_n)$ be the sequence in $(0,\infty)$ such that $(x_n)$ is strictly decreasing and converges to $0$, and $(f(x_n))$ converges to $f(0) = 0$. Therefore if $\varepsilon > 0$ is given, there exists $N\in \mathbb N$ such that 
$$
f(x_N) = f(x_N)-f(0) < \varepsilon.
$$
This proves 3). Since 3) and 2) are equivalent, we see that 2), 5), 6), 7) and 3) are equivalent, as asserted.\vspace{0.3cm}\\
\indent It is true in general that every continuous function is quasi continuous. So it is easy to see that 1) implies 8) and 8) implies 9). Next assume that 9) holds. To show 3), let $\varepsilon > 0$ be given. Let $V = U = [0,\varepsilon)$. Then $V$ and $U$ are open set in $[0,\infty)$ containing $f(0) = 0$ and $0$, respectively. Since $f$ is quasi continuous at $0$, there exists a nonempty open set $G\subseteq U$ such that $f(G)\subseteq V$. Now we can choose $x\in G-\{0\}$ so that $x>0$ and $f(x)<\varepsilon$. This gives 3). Since 1) and 3) are equivalent, we obtain that 1), 8), 9), and 3) are equivalent. Similarly, it is easy to see that 1) implies 10), 10) implies 11), 1) implies 12), and 12) implies 13). Since 1) and 3) are equivalent, it now suffices to show that each of 11) and 13) implies 3). First assume that 11) holds. Let $\varepsilon > 0$ and let $V = [0,\varepsilon)$. Then $V$ is open in $[0,\infty)$ and contains $f(0)$. Since $f$ is a.c.S. at $0$, there exists an open set $U$ containing $0$ such that 
$$
f(U) \subseteq \inte(\overline V) = \inte [0,\varepsilon] = [0,\varepsilon). 
$$ 
Now we can choose $x\in U-\{0\}$ so that $x>0$ and $f(x)<\varepsilon$. Similarly if 13) holds, then $\overline{f^{-1}(V)}$ is a neighborhood of $0$, so $f^{-1}(V)\neq \{0\}$, and therefore we can choose $x\in f^{-1}(V)-\{0\}$ so that $f(x)<\varepsilon$ and $x>0$. This completes the proof.
\end{proof}
The function $f$ in Example \ref{example24fdisnotmetric} shows that in the case of ultrametric-metric-preserving functions, the global continuity on $[0,\infty)$ and the local continuity at $0$ are not equivalent. In addition, the uniform continuity on $[0,\infty)$ and continuity on $[0,\infty)$ are not equivalent as can be seen from the function $f$ in Example \ref{exam11sectionultrametricpreserving}. However, we still have the following result for the continuity at $0$. 
\begin{thm}\label{equvalentultrametricthm}
Let $f$ be ultrametric-metric-preserving. Then the following statements are equivalent:
\begin{itemize}
\item[(i)] $f$ is continuous at $0$,
\item[(ii)] $f$ is weakly continuous at $0$,
\item[(iii)] for every $\varepsilon > 0$, there exists an $x>0$ such that $f(x) < \varepsilon$,
\item[(iv)] $f$ is quasi continuous at $0$,
\item[(v)] $f$ is a.c.S. at $0$,
\item[(vi)] $f$ is a.c.H. at $0$.
\end{itemize}
\end{thm}
\begin{proof}
We have that (i) implies (ii) is true in general. By the same argument that 7) implies 3) in Theorem \ref{newthmcontinuitythm}, we see that (ii) implies (iii). Next assume that (iii) holds. To show that $f$ is continuous at $0$, let $\varepsilon > 0$ be given. Then by (iii), there exists $x_0 > 0$ such that $f(x_0) < \frac{\varepsilon}{2}$. Let $\delta = x_0$ and let $x\in [0,\delta)$. Since $0\leq x <\delta$ and $f\in \mathcal{UM}$, we obtain by Corollary \ref{corafteramenablelemma} and Theorem \ref{ultrametricthm} that 
$$
\left|f(x)-f(0)\right| = f(x) \leq 2f(\delta) = 2f(x_0) < \varepsilon.
$$
This gives (i). Therefore (i), (ii), and (iii) are equivalent. Since (i) implies (iv), (v), and (vi), it suffices to show that each of (iv), (v), and (vi) implies (iii). Since $f\in \mathcal{UM}$, it is amenable and we can use the same argument of the proof of Theorem \ref{newthmcontinuitythm} to show that (iv) implies (iii) (the same as 9) implies 3)), (v) implies (iii) (the same as 11) implies 3)), and (vi) implies (iii) (the same as 13) implies 3)). This completes the proof. 
\end{proof}
\begin{cor}
Let $f\in \mathcal{UM}$. If $f$ is discontinuous at $0$, then there exists an $\varepsilon > 0$ such that $f(x) > \varepsilon$ for all $x>0$. 
\end{cor}
\begin{proof}
This follows from (i) and (iii) in Theorem \ref{equvalentultrametricthm}.
\end{proof}
\begin{exam}
Let $f, g:[0,\infty)\to[0,\infty)$ be given by
$$
f(x) = \begin{cases}
x,\quad&\text{$x\leq 1$};\\
1,\quad&\text{$x > 1$ and $x\in \mathbb Q$};\\
2,\quad&\text{$x > 1$ and $x\notin \mathbb Q$},
\end{cases}\quad\quad
g(x) = \begin{cases}
x, \quad &x<1;\\
2, \quad &x\geq 1.
\end{cases}
$$
\end{exam}
First we will show that $f\in \mathcal{UM}$ by applying Theorem \ref{ultrametricthm}. So we let $0\leq a\leq b$. If $b>1$, then $2f(b) \geq 2 \geq f(x)$ for every $x\in [0,\infty)$. In particular, $2f(b)\geq f(a)$. If $b\leq 1$, then $f(a) = a \leq b \leq 2b = 2f(b)$. So $f\in \mathcal{UM}$. It is easy to see that $f$ is weakly continuous at $1$ but is not continuous at $1$. In fact $f$ is weakly continuous at every $x\geq 0$ and is not continuous at any $x\geq 1$. This shows that we cannot replace continuity at $0$ in Theorem \ref{equvalentultrametricthm} by continuity at any other point $x\neq 0$. Similarly, $g\in \mathcal{UM}$ and quasi continuous on $[0,\infty)$ but $g$ is not continuous at $1$.\\[0.3cm]
\noindent \textbf{Acknowledgment}\\[0.2cm]
The first author receives financial support from The Thailand Research Fund, research grant number TRG5680052. The author takes this opportunity to thanks The Thailand Research Fund for the support.

\end{document}